\UseRawInputEncoding
\documentclass[12pt]{article}

\usepackage{amsfonts,amsmath,amsthm}
\usepackage{amssymb}
\usepackage{algorithm}
\usepackage{algpseudocode}
\usepackage{xcolor}
\usepackage{graphicx}
\usepackage{stmaryrd}        
\usepackage{multirow}
\usepackage{dsfont}

\usepackage[paper=a4paper,dvips,top=2cm,left=2cm,right=2cm,
foot=1cm,bottom=4cm]{geometry}

\newcommand{\red}[1]{\textcolor{red}{#1}}

\newcommand{\vc}{{\bf c}}

\newcommand{\vx}{{\bf x}}
\newcommand{\vy}{{\bf y}}

\newcommand{\vz}{{\bf z}}

\newcommand{\R}{\mathbb{R}}

\newtheorem{Thm}{Theorem}[section]
\newtheorem{Def}[Thm]{Definition}

\newtheorem{remark}[Thm]{Remark}

\begin{document}
	
\title{The Sum of Squares Rank of Biquadratic Forms and The Zarankiewicz Number}
	\Large
	\author{Chunfeng Cui\footnote{School of Mathematical Sciences, Beihang University, Beijing  100191, China.
			{\tt chunfengcui@buaa.edu.cn})}
		\and
		Liqun Qi\footnote{
			Department of Applied Mathematics, The Hong Kong Polytechnic University, Hung Hom, Kowloon, Hong Kong.
			({\tt maqilq@polyu.edu.hk})}
		\and {and \
			Yi Xu\footnote{School of Mathematics, Southeast University, Nanjing  211189, China. Nanjing Center for Applied Mathematics, Nanjing 211135,  China. Jiangsu Provincial Scientific Research Center of Applied Mathematics, Nanjing 211189, China. ({\tt yi.xu1983@hotmail.com})}
		}
	}

	\date{\today}
	\maketitle

     \begin{abstract}
	Denote the maximum sos rank of $m \times n$ sum of squares (SOS) biquadratic forms by $BSR(m, n)$. In this paper, we show that $BSR(m, n) \ge z(m, n)$ and conjecture that $BSR(m, n) = z(m, n)$, where $z(m, n)$ is the Zarankiewicz number. Our result coincides with the existing results for $m = 2$, $n = 2$, and $m = n = 3$, and is superior to other previously known lower bounds. Our result also connects graph theory and SOS polynomial theory.
	
	\medskip
	
	\textbf{Keywords.} Biquadratic forms, sum-of-squares, positive semi-definiteness, sos rank, bipartite graphs, the Zarankiewicz number.
	
	\medskip
	\textbf{AMS subject classifications.} 14P10, 05C35, 11E25, 15A69, 90C22.
\end{abstract}

	\section{Introduction}

Let $m, n \ge 2$.   Denote $\{1, \ldots, m\}$ as $[m]$.   An $m \times n$ biquadratic form is a quartic homogeneous polynomial of $m+n$ variables, with the form
\begin{equation} \label{e0}
P(\mathbf{x}, \mathbf{y}) = \sum_{i,k=1}^{m} \sum_{j,l=1}^{n} a_{ijkl} x_i x_k y_j y_l,
\end{equation}
where
\begin{equation} \label{e1}
a_{ijkl} = a_{kjil} = a_{klij}
\end{equation}
for $i, k \in [m], j, l \in [n]$, are real numbers.
If $P(\vx, \vy) \ge 0$ for all $\vx \in \Re^m$ and $\vy \in \Re^n$, then we say that $P$ is positive semi-definite (PSD), or nonnegative.   If
\begin{equation} \label{e2}
P(\vx, \vy) = \sum_{p=1}^r f_p(\vx, \vy)^2,
\end{equation}
where each $f_p$ for $p \in [r]$ is an $m \times n$ bilinear form, then $P$ is said to be a sum of squares (SOS).   The smallest value of $r$, such that (\ref{e2}) holds, is called the sos rank of $P$.    By symmetry, we may assume that $m \ge n$ to simplify our discussion.

A PSD biquadratic form may not be SOS.  In 1973, Calder\'{o}n \cite{Ca73} showed that all $m \times 2$ or $2 \times n$ PSD biquadratic forms are SOS.   In 1975, Choi \cite{Ch75} presented a $3 \times 3$ PSD biquadratic form which is not SOS.
In \cite{QCX26}, the following definition was introduced:

\begin{Def}
Let $m, n \ge 2$.  Let $BSR(m, n)$ be the maximum sos rank of $m \times n$ SOS biquadratic forms.
\end{Def}

\begin{itemize}
    \item It is known that \(BSR(m,2) = m+1\). See \cite{BPSV19}.
    \item It is known that \(BSR(3,3) = 6\). See \cite{BSSV22}.
\end{itemize}

In this paper, we show that $BSR(m, n) \ge z(m, n)$ and conjecture that $BSR(m, n) = z(m, n)$, where $z(m, n)$ is the Zarankiewicz number.

In the next three sections, we present some preliminary knowledge on the sos rank of biquadratic forms, and the Zarankiewicz number.   In Section 5, we present the main theorem of this paper.

\section{The Gram Matrices of A Biquadratic Form}
\label{sec:Gram-matrices}

 The Gram matrix method was first developed in \cite{CLR95}.  Let \(P(\vx, \vy)\) be an \(m \times n\) biquadratic form and \(\vz = \vx \otimes \vy \in \R^{mn}\). Then we may write
\(P(\vx, \vy) = \vz^\top M \vz\), where \(M\) is an \(mn \times mn\) matrix. The matrix \(M\) is called a Gram matrix of \(P\). It is not unique. The Gram matrix is named after Danish mathematician J\/{o}rgen Pedersen Gram (1850-1916). We may index \(\vz\) by \((i, j)\) and \((k, l)\). Then a natural Gram matrix is \(M_0 = (a_{ijkl})\).

The non-uniqueness of Gram matrices is important. For any skew-symmetric matrix \(S\) (i.e., \(S^\top = -S\)), adding it to \(M\) does not change the quadratic form because \(\vz^\top S \vz = 0\) for all \(\vz\). More generally, any matrix \(H\) such that \(\vz^\top H \vz = 0\) for all \(\vz = \vx \otimes \vy\) can be added. These matrices form a linear subspace.

If we index \(\vz\) by the pairs \((i,j)\) with \(i\in[m]\) and \(j\in[n]\), then the entries of \(M\) are naturally indexed by \(((i,j),(k,l))\). The symmetry conditions \(a_{ijkl} = a_{kjil} = a_{klij}\) from \eqref{e1} ensure that the natural Gram matrix \(M_0\) is symmetric.

\medskip

\noindent\textbf{SOS and PSD Gram Matrices.}
A fundamental characterization relates SOS decompositions to positive semidefinite Gram matrices:

\begin{quote}
\emph{\(P\) is SOS if and only if there exists a positive semidefinite (PSD) Gram matrix representing \(P\).}
\end{quote}

To see why, first suppose \(P\) is SOS: \(P(\vx,\vy) = \sum_{p=1}^r f_p(\vx,\vy)^2\) with each \(f_p\) bilinear. Write \(f_p(\vx,\vy) = \vc_p^\top \vz\) for some vectors \(\vc_p \in \R^{mn}\). Then
\[
P = \sum_{p=1}^r (\vc_p^\top \vz)^2 = \vz^\top \left( \sum_{p=1}^r \vc_p \vc_p^\top \right) \vz,
\]
and \(M = \sum_{p=1}^r \vc_p \vc_p^\top\) is symmetric and PSD (a sum of rank-1 PSD matrices). Thus a PSD Gram matrix exists.

Conversely, suppose there exists a PSD Gram matrix \(M\) with \(P(\vx,\vy) = \vz^\top M \vz\). Since \(M \succeq 0\), it admits a Cholesky factorization \(M = C^\top C\) for some matrix \(C \in \R^{r \times mn}\). Let \(\vc_1,\dots,\vc_r \in \R^{mn}\) be the rows of \(C\). Then
\[
P = \vz^\top C^\top C \vz = \sum_{p=1}^r (\vc_p^\top \vz)^2 = \sum_{p=1}^r f_p(\vx,\vy)^2,
\]
where each \(f_p\) is bilinear because \(\vz = \vx \otimes \vy\). Hence \(P\) is SOS, and the number of squares \(r\) is at least the rank of \(M\).

\medskip

\noindent\textbf{Important Nuance: Existence vs. Universality.}
The characterization above states that \emph{there exists} a PSD Gram matrix if and only if \(P\) is SOS. It does \emph{not} say that every Gram matrix representing \(P\) is PSD. In fact, most Gram matrices are not PSD.

The set of all Gram matrices for a fixed \(P\) forms an affine space:
\[
\mathcal{M}_P = \{ M(\Gamma) : \Gamma \} = M_0 + \mathcal{H},
\]
where \(M_0\) is a particular Gram matrix (e.g., the natural one) and \(\mathcal{H}\) is a linear subspace consisting of matrices \(H\) with \(\vz^\top H \vz = 0\) for all \(\vz = \vx \otimes \vy\). The dimension of this affine space is \(\binom{m}{2}\binom{n}{2}\), corresponding to the degrees of freedom in representing the cross terms \(x_i x_k y_j y_l\) with \(i \neq k\) and \(j \neq l\).

If \(P\) is SOS, the intersection \(\mathcal{M}_P \cap \{ \text{PSD matrices} \}\) is nonempty¡ªthere is at least one PSD Gram matrix. However, because we can add skew-symmetric matrices or other nullity matrices to any Gram matrix, most points in \(\mathcal{M}_P\) lie outside the PSD cone. This leads to the following theorem:

\begin{Thm} ({\bf Theorem 7 of \cite{QCX26}}) \label{t6.1}
		Suppose that the $m\times n$ biquadratic form in (\ref{e0}) is SOS. Then we have
			\begin{equation}\label{def:SOS_rank}
		\operatorname{SOS\text{-}rank}(P)=\min_{\Gamma: M(\Gamma)\succeq 0}\;\operatorname{rank}(M(\Gamma)).
		\end{equation}
	\end{Thm}

We must search through the affine space \(\mathcal{M}_P\) to find the PSD Gram matrices, and among those, find one of minimal rank. The rank of this optimal PSD Gram matrix is exactly the SOS rank.

\medskip

\noindent\textbf{Geometric Intuition.}
Think of \(\mathcal{M}_P\) as a line (or higher-dimensional plane) in the space of symmetric matrices. The PSD matrices form a convex cone. If \(P\) is SOS, this line intersects the PSD cone¡ªperhaps at a single point, perhaps along a segment. Points on the line away from this intersection are not PSD, yet they still represent the same polynomial \(P\).

\medskip

\noindent\textbf{Example 1: An SOS form with a non-PSD Gram matrix.}
Consider the simple \(2 \times 2\) biquadratic form
\[
P(\vx,\vy) = x_1^2 y_1^2 + x_2^2 y_2^2.
\]
Clearly \(P\) is SOS because \(P = (x_1 y_1)^2 + (x_2 y_2)^2\). With \\
\(\vz = (x_1 y_1, x_1 y_2, x_2 y_1, x_2 y_2)^\top\), one natural Gram matrix is the diagonal matrix
\[
M_0 = \begin{pmatrix}
1 & 0 & 0 & 0 \\
0 & 0 & 0 & 0 \\
0 & 0 & 0 & 0 \\
0 & 0 & 0 & 1
\end{pmatrix},
\]
which is PSD. Now add a skew-symmetric matrix
\[
S = \begin{pmatrix}
0 & 0 & 0 & 0 \\
0 & 0 & 1 & 0 \\
0 & -1 & 0 & 0 \\
0 & 0 & 0 & 0
\end{pmatrix}.
\]
Since \(S^\top = -S\), we have \(\vz^\top S \vz = 0\) for all \(\vz\); hence \(M = M_0 + S\) is also a Gram matrix for \(P\):
\[
M = \begin{pmatrix}
1 & 0 & 0 & 0 \\
0 & 0 & 1 & 0 \\
0 & -1 & 0 & 0 \\
0 & 0 & 0 & 1
\end{pmatrix}.
\]
Checking the \(2 \times 2\) principal submatrix for indices \((2,3)\) gives \(\begin{pmatrix}0 & 1 \\ -1 & 0\end{pmatrix}\), whose determinant is \(-1\); therefore \(M\) is not positive semidefinite. Thus we have exhibited a Gram matrix for an SOS form that is not PSD, illustrating that the existence of a PSD Gram matrix does not imply that all Gram matrices are PSD.

\medskip

\noindent\textbf{Example 2: A PSD non-SOS form (Choi's example) has no PSD Gram matrix.}
Choi \cite{Ch75} discovered the following \(3 \times 3\) biquadratic form:
\[
C(\vx,\vy) = x_1^2 y_1^2 + x_2^2 y_2^2 + x_3^2 y_3^2 + x_1^2 y_2^2 + x_2^2 y_3^2 + x_3^2 y_1^2 - \big( x_1^2 y_3^2 + x_2^2 y_1^2 + x_3^2 y_2^2 \big).
\]
This form is known to be PSD (it is nonnegative for all real \(\vx,\vy\)) but it is \emph{not} SOS. By the fundamental characterization, if there existed a PSD Gram matrix for \(C\), then \(C\) would be SOS. Since \(C\) is not SOS, it follows that \emph{every} Gram matrix representing \(C\) fails to be PSD. Hence the affine space \(\mathcal{M}_C\) contains no positive semidefinite matrices at all. This example shows that for a non-SOS PSD form, the intersection \(\mathcal{M}_P \cap \{\text{PSD matrices}\}\) is empty.

\section{Existing Lower Bounds for BSR(m,n)}

Let \(m \ge n\). We say a biquadratic form is a \textbf{simple biquadratic form} if it contains only distinct terms of the type \(x_i^2 y_j^2\).

For comparison, we recall a basic lower bound that follows from elementary considerations of simple biquadratic forms (see \cite{XCQ26}, Theorem 2.10):
\[
\mathrm{BSR}_{\mathrm{simple}}(m,n) \ge m + n \quad \text{for all } m,n \ge 3.
\]
Consequently, \(\mathrm{BSR}(m,n) \ge m + n\). However, as we shall see in Section 5, this bound is not tight; the Zarankiewicz number \(z(m,n)\) provides a much stronger lower bound.


\section{The Zarankiewicz Number}

Kazimierz Zarankiewicz (1902--1959) was a Polish mathematician and Professor at the Warsaw University of Technology who was interested primarily in topology and graph theory.

\subsection{Definition and Basic Concepts}

The problem of determining the Zarankiewicz number, named after its proposer \cite{Za51}, is a classical problem in extremal graph theory \cite{Bo04}. The Zarankiewicz number \(z(m,n;s,t)\) denotes the maximum possible number of edges in a bipartite graph with parts of sizes \(m\) and \(n\) that does not contain a complete bipartite subgraph \(K_{s,t}\) (with \(s\) vertices in the first part and \(t\) vertices in the second part) as a subgraph \red{\cite{DHS13}.}

In this paper, we are primarily concerned with the case \(s = t = 2\), i.e., forbidding \(K_{2,2}\), which is simply a cycle of length four (a 4-cycle). For brevity, the notation \(z(m,n)\) is often used to denote \(z(m,n;2,2)\).

\begin{Def}
Let \(G = (S,T,E)\) be a bipartite graph with \(|S| = m\) and \(|T| = n\). We say that \(G\) is \(K_{2,2}\)-free (or 4-cycle-free) if it contains no subgraph isomorphic to \(K_{2,2}\). Then
\[
\begin{aligned}
z(m,n) = & z(m,n;2,2)\\
 = & \max\{ |E| : \; G \text{ is a } K_{2,2}\text{-free} \\
&\; \text{bipartite graph with part sizes } m \text{ and } n \}.
\end{aligned}
\]

\end{Def}

\begin{remark}\label{rem:growth}
For small parameters, \(z(m,n)\) may be close to \(m+n\) (e.g., \(z(3,3)=6\), \(z(4,3)=7\)).
However, for larger values the Zarankiewicz number grows faster than linearly.
The classic K\H{o}v\'ari--S\'os--Tur\'an theorem  \cite{KST54}
gives the general upper bound
{$z(p^2+p, p^2)=p^2(p+1)+1$ for prime numbers $p$.}
Constructions using finite projective planes, pioneered by
Reiman \cite{RS65}, show that this bound is tight for infinite families.
In particular,
{
\[
z(m,n) \leq \frac{n}{2} + \frac{1}{2}\sqrt{n^2  + 4mn(m-1)}+1.
\]
}
\end{remark}

\subsection{Known Bounds and Exact Values}

Determining Zarankiewicz numbers exactly is a difficult problem in combinatorics, and exact values are known only for small parameters {\cite{Gu69, RS65}.} 
Some classical results include:

\begin{itemize}
    \item \textbf{Small exact values:}
    \begin{itemize}
        \item \(z(3,3) = 6\) (achieved by the 6-cycle, i.e., the incidence graph of the projective plane of order 2); 
        \item \(z(3,4) = z(4,3) = 7\); 
        \item \(z(4,4) = 9\); 
        \item \(z(5,4) = 10\); 
        \item \(z(5,5) = 12\); 
    \end{itemize}

    \item \textbf{General bounds:} The classic and foundational result in this area is {by Reiman \cite{RS65},}
    \[
    z(m,n) \leq \frac{n}{2} + {\frac{1}{2}\sqrt{n^2  + 4mn(m-1)}+1.}
    \]
    For fixed \(n\) and large \(m\), this theorem yields the
     asymptotic bound:
    \[
    z(m,n) = O(m^{1/2}n + n).
    \]

    \item \textbf{Recent improvements:} Recent work has provided improved upper bounds for many small parameter sets and established new families of closed-form bounds \cite{CHM24,  DHS13}.
\end{itemize}

\subsection{Connection to Biquadratic Forms}

The connection between Zarankiewicz numbers and biquadratic forms arises from the following observation: if a bipartite graph \(G\) contains no 4-cycle, then the corresponding simple biquadratic form \(P_G(\mathbf{x},\mathbf{y}) = \sum_{(i,j)\in E} x_i^2 y_j^2\) has the property that its SOS rank equals the number of edges \(|E|\). This leads naturally to the conjecture that \(\mathrm{BSR}(m,n) = z(m,n)\).

\begin{remark}
The relationship between extremal graph theory and polynomial optimization, particularly through sum-of-squares representations, has emerged as a fruitful area of research. The Zarankiewicz number provides a combinatorial interpretation for the maximum SOS rank of simple biquadratic forms, and further exploration of this connection may yield new insights into both fields.
\end{remark}

\section{The Zarankiewicz Number and The Lower Bounds of $BSR(m, n)$}

Suppose that \(m \ge n \ge 3\). Let \(S = \{1,\dots,m\}\) and \(T = \{1,\dots,n\}\). Consider a bipartite graph \(G = (S,T,E)\), where \(S\) and \(T\) are the two vertex sets and \(E\) is the edge set. For such a graph, we define an \(m\times n\) simple biquadratic form
\[
P_G(\mathbf{x},\mathbf{y}) = \sum_{(i,j)\in E} x_i^2 y_j^2. \tag{5.1}
\]

We have the following theorem.

\begin{Thm} \label{thm:graph-sos}
If \(G\) contains no \(4\)-cycle, then \(\operatorname{SOS-rank}(P_G)=|E|\).
\end{Thm}

\begin{proof}
The upper bound \(\operatorname{SOS-rank}(P_G)\le |E|\) is immediate because \(P_G\) itself is a sum of \(|E|\) squares \((x_i y_j)^2\).

For the lower bound, assume \(P_G = \sum_{t=1}^{r} L_t^2\) with each \(L_t\) a bilinear form
\(L_t = \sum_{i,j} c_{ij}^{(t)} x_i y_j\).
For every edge \((i,j)\in E\) define the vector
\[
\mathbf{v}_{ij} = (c_{ij}^{(1)},\dots,c_{ij}^{(r)}) \in \mathbb{R}^r.
\]
From the coefficient of \(x_i^2 y_j^2\) we obtain \(\|\mathbf{v}_{ij}\|^2 = 1\). For any pair \((i,j)\notin E\), the coefficient of \(x_i^2 y_j^2\) in \(P_G\) is zero, so \(\sum_{t=1}^r (c_{ij}^{(t)})^2 = 0\), which forces \(c_{ij}^{(t)} = 0\) for all \(t\); hence \(\mathbf{v}_{ij} = \mathbf{0}\).

Now consider any mixed monomial \(x_i x_k y_j y_l\) with \(i\neq k\) or \(j\neq l\). In \(\sum_t L_t^2\) its coefficient equals
\[
2\sum_{t} c_{ij}^{(t)}c_{kl}^{(t)} + 2\sum_{t} c_{il}^{(t)}c_{kj}^{(t)} = 2\bigl(\mathbf{v}_{ij}\!\cdot\!\mathbf{v}_{kl} + \mathbf{v}_{il}\!\cdot\!\mathbf{v}_{kj}\bigr).
\]
Because \(P_G\) contains no such mixed term, we must have
\begin{equation} \label{eq:cycle-condition}
\mathbf{v}_{ij}\!\cdot\!\mathbf{v}_{kl} + \mathbf{v}_{il}\!\cdot\!\mathbf{v}_{kj} = 0
\qquad\text{whenever }(i,j),(k,l)\in E.
\end{equation}

If \(G\) contains no \(4\)-cycle, then for any distinct edges \((i,j)\) and \((k,l)\) with \(i\neq k\) and \(j\neq l\), at most one of the two pairs \(\{(i,j),(k,l)\}\) and \(\{(i,l),(k,j)\}\) can consist entirely of edges; otherwise the four vertices would form a \(4\)-cycle \(i\!-\!j\!-\!k\!-\!l\!-\!i\). Therefore, in \eqref{eq:cycle-condition}, one of the two dot products involves a vector corresponding to a non?edge, which is the zero vector. Hence the equation reduces to \(\mathbf{v}_{ij}\!\cdot\!\mathbf{v}_{kl}=0\) for all distinct edges with \(i\neq k,\;j\neq l\).

For the cases where edges share a vertex (\(i=k\) or \(j=l\)), we obtain directly from \eqref{eq:cycle-condition} by setting appropriate indices:
\[
\mathbf{v}_{ij}\!\cdot\!\mathbf{v}_{il}=0\;(j\neq l),\qquad
\mathbf{v}_{ij}\!\cdot\!\mathbf{v}_{kj}=0\;(i\neq k).
\]

Thus all vectors \(\{\mathbf{v}_{ij} : (i,j)\in E\}\) are pairwise orthogonal and each has unit length. Consequently \(r \ge |E|\), i.e. \(\operatorname{SOS-rank}(P_G) \ge |E|\). Together with the upper bound we obtain equality.
\end{proof}

\medskip

\noindent\textbf{Conjecture.}
For \(m,n\ge 2\) let \(z(m,n)\) denote the maximum number of edges in a bipartite graph with parts of sizes \(m\) and \(n\) that contains no \(4\)-cycle (the Zarankiewicz number \(z(m,n;2,2)\)). We conjecture that
\[
\operatorname{BSR}(m,n) = z(m,n).
\]

\noindent\textbf{Examples and remarks.}
\begin{itemize}
    \item For \((m,n)=(3,3)\) we have \(z(3,3)=6\) (e.g. the graph with edges
    \((1,1),(2,2),(3,3),(1,2),(2,3),(3,1)\)). This matches the known result \(\operatorname{BSR}(3,3)=6\).

    \item For \((m,n)=(4,3)\) the Zarankiewicz number is \(z(4,3)=7\).
    A concrete \(C_4\)-free graph achieving this bound has parts
    \(S=\{1,2,3,4\}\), \(T=\{1,2,3\}\) and edge set
    \[
    E = \{(1,1),(2,1),(3,1),\;(1,2),(4,2),\;(2,3),(4,3)\}.
    \]
    (One verifies directly that no two vertices in \(S\) share two common neighbours in \(T\).)
    Hence there exists a simple biquadratic form on \(4\times 3\) variables with SOS rank \(7\).
    Consequently \(\operatorname{BSR}(4,3) \ge 7\).
    This matches the elementary bound \(m+n = 7\) from Section 3.

    \item Whether \(\operatorname{BSR}(4,3)\) equals \(7\) or can be larger remains an open question.
    If the conjecture holds, then \(\operatorname{BSR}(4,3)=7\); otherwise there exist non-simple forms
    with SOS rank strictly greater than the Zarankiewicz number.
\end{itemize}

\medskip
	
	\bigskip
	
	\noindent\textbf{Acknowledgement}
	I am thankful to Professors Greg Blekherman and Guangzhou Chen for the discussions.  This work was partially supported by Research Center for Intelligent Operations Research, The Hong Kong Polytechnic University (4-ZZT8), the National Natural Science Foundation of China (Nos. 12471282 and 12131004),  and Jiangsu Provincial Scientific Research Center of Applied Mathematics (Grant No. BK20233002).
	
	\medskip
	
	\noindent\textbf{Data availability}
	No datasets were generated or analysed during the current study.
	
	\medskip
	
	\noindent\textbf{Conflict of interest} The authors declare no conflict of interest.

\end{document}